\documentclass[10pt]{article}
\usepackage[utf8]{inputenc}
 
\usepackage{amssymb}
\usepackage{amsmath}
\usepackage{amsthm}
\usepackage{url}
\usepackage{mathrsfs}
\usepackage{stmaryrd}
\usepackage{enumitem}
\usepackage{mathtools}

\newtheorem{prop}{Proposition}[section]
\newtheorem{lem}{Lemma}[section]

\newtheorem{theor}{Theorem}[section]
\newtheorem{mytheor}{Theorem}

\newtheorem{cor}{Corolary}[section]
\newtheorem{conj}{Conjecture}[section]

\newcommand{\Uni}{\mathcal{U}}

\newcommand{\Z}{\mathbb{Z}}
\newcommand{\Q}{\mathbb{Q}}
\newcommand{\C}{\mathbb{C}}
\newcommand{\T}{\mathbb{T}}

\newcommand{\bop}{\mathcal{B}}

\newcommand{\Hi}{\mathcal{H}}

\newcommand{\bp}{\bar{\pi}}

\DeclareMathOperator*{\esssup}{ess\,sup}

\DeclareMathOperator{\Aut}{Aut}
\DeclareMathOperator{\Rep}{Rep}
\DeclareMathOperator{\QRep}{QRep}

\DeclareMathOperator{\Id}{Id}

\DeclareMathOperator{\rp}{rp}
\DeclareMathOperator{\en}{en}
\DeclareMathOperator{\df}{df}

\begin{document}

\author{Andrei Alpeev  \footnote{Einstein Institute of Mathematics, Edmond J. Safra Campus (Givat Ram), The Hebrew University, Jerusalem 91904, Israel} \footnote{Euler Mathematical Institute at St. Petersburg State University alpeevandrey@gmail.com}}
\title{Lamplighters over non-amenable groups are not strongly Ulam stable}


\maketitle
\begin{abstract}
\end{abstract}

\section{Introduction}

Let $G$ be a countable group. A representation of $G$ is any homomorphism of this group into the group of invertible bounded operators on a Hilbert space. A representation is called unitary if the range of this homomorphism is a subset of the group of unitary operators. We will say that a representation is {\em unitarizable} if there is another scalar product on the Hilbert space that would generate the same topology and that the representation is unitary with respect to this new Hilbert structure. For our purposes it would be convenient to give an equivalent definition. A representation $\pi$ of a countable group $G$ on a Hilbert space $\Hi$ is unitarizable if there is a bounded operator $B$ on $\Hi$ and a constant $C > 0$ such that $(\pi(g))^* B \pi(g) = B$ for all $g \in G$, and $\langle B v , v\rangle \geq C \lVert v\rVert^2$ for all $v \in \Hi$. A representation $\pi$ of a countable group $G$ is called uniformly bounded if there is a constant $K$ such that $\lvert \pi(g) \rvert < K$ for all $g \in G$. A countable group is called {\em unitarizable} if each of its uniformly bounded representations is unitarizable. 

These natural definitions rise the question: which countable groups are unitarizable. Early on, in the 50's, the following answer was proposed by Dixmier \cite{D50}:
\begin{conj}

A countable group is unitarizable iff it is amenable.

\end{conj}
The ``if'' direction was handled by Dixmier in the same paper, by M. Day \cite{Da50}, and by Nakamura and Takeda \cite{NT51}.
The other direction proved to be much harder.
It was shown that the free group $F_2$ on two generators is not unitarizable(see e.g. \cite{P04}). 

\begin{theor}

The free group $F_2$ has a uniformly bounded non-unitarizable representation on a separable Hilbert space.

\end{theor}

For some more recent results on non-unitarizable groups see \cite{EM09}, \cite{S15}, \cite{Geta20}. 

Using an induced representation construction it is easy to prove the following well-known result:

\begin{prop}\label{prop: subgoups of unitarizable}

A subgroup of a unitarizable group is unitarizable.

\end{prop}

So, if only the von Neumann--Day conjecture (stating that all non-amenable groups contain a copy of a non-abelian free subgroup) was true, the Dixmier problem would be solved, but it is known from the work of Olshanski \cite{Ol91} that there are non-amenable groups that do not contain non-abelian free subgroups. While the von Neumann--Day conjecture is not true in its original form, its ``measured'' counterpart holds true. We would say that a group $G$ orbitally contains a group $N$ if there is a standard probability space $(X, \mu)$ and essentially free measure preserving actions of $G$ and $N$ on $X$ such that the orbit of almost every point in $X$ under the $N$-action  is a subsets of the orbit under the $G$-action. We would call such a pair of actions an {\em orbit inclusion} of $N$ into $G$. We remind that a Bernoulli action of a countable group $G$ is a shift-action of $G$ on the product space $(K, \kappa)^G = (\prod_{g\in G} K, \bigotimes_{g \in G} \kappa)$, where $(K, \kappa)$ is a  standard probability space.

\begin{theor}[Gaboriau-Lyons, \cite{GL09}]

Any non-amenable countable group $G$ orbitally contains the free group $F_2$ on two generators. In fact, the action of $G$ for the orbit inclusion could be taken to be the Bernoulli action $G \curvearrowright [0,1]^G$, where $[0,1]^G$ is endowed with the product of standard Lebesgue measures on $[0,1]$.

\end{theor}

Using the latter theorem, Monod and Ozawa proved that a large class of groups are non-unitarizable:

\begin{theor}[Monod-Ozawa, \cite{MO09}]
If $G$ is a countable non-amenable group and $A$ is any countable group containing an infinite abelian subgroup, then the restricted wreath product $A \wr G$ is non-unitarizable.
\end{theor}

We remind that the wreath product is defined as follows:

\[
A \wr G = \bigoplus_G A \rtimes G,
\]
where $\bigoplus_G A$ is the direct sum of copies of $A$ indexed by $G$ upon which $G$ acts by shifts.

Lewis Bowen provided the following refinement of the Gaboriau-Lyons theorem:

\begin{theor}[L. Bowen, \cite{B19}]\label{thm: Bernoulli has Day property}
In the setting of the Gaboriau-Lyons theorem, the action of $G$ could be taken to be any non-trivial Bernoulli action $G \curvearrowright (K, \kappa)^G$. 
\end{theor}

In this paper we use Bowen's result to improve upon the Monod-Ozawa theorem by dropping the infinite abelian subgroup requirement:

\begin{mytheor}\label{th: main}
If $G$ is a countable non-amenable group and $A$ is any non-trivial countable group, then the restricted wreath product $A \wr G$ is non-unitarizable. In particular, the standard lamplighter group $(\Z/2\Z) \wr G$ over any non-amenable group is non-unitarizable.
\end{mytheor}

In light of Proposition \ref{prop: subgoups of unitarizable}, it is enough to prove the theorem assuming $A$ to be a cyclic group (note that $A' \wr G \leqslant  A'' \wr G$ if $A' \leqslant  A''$). In fact, since unitarizability obviously passes to factor-groups, we may assume that $A$ is a finite cyclic group.

On the high level, the proof resembles that of Proposition \ref{prop: subgoups of unitarizable}. Intuitively, instead of constructing induced group representation, we construct an induced representations of measurable equivalence relations. This construction is the same as one used in the paper by Ozawa and Monod. A technical difference is that they used reformulation of unitarizability in terms of derivations, whereas the proof presented here relies on reformulation of unitarizability in terms of existence of a unitarizing operator.

We next tackle the question of strong Ulam stability, originally posed by Ulam in \cite{U60}. A map $\rho: G \to \Uni(\Hi)$ is called an $\varepsilon$-representation if $\rho(\Id_G) = \Id_{\Hi}$ and $\lVert \rho(g)\rho(h) - \rho(gh)\rVert < \varepsilon$ for all $g,h \in G$. On the space of all maps $G \to \Uni(\Hi)$ we may define the uniform metric: $\lVert \rho_1 - \rho_2 \rVert = \sup_{g \in G} \lVert \rho_1(g) = \rho_2(g)\rVert$. We would say that a group $G$ is strongly Ulam stable if for every $\varepsilon > 0$ there is a $\delta > 0$ such that for any $\delta$-representation $\rho$ there an (actual) represnetation $\rho'$ in the distance less than $\varepsilon$. It was proved by Kazhdan \cite{K82} that amenable groups are strongly Ulam stable. In the same paper he  showed that surface groups are not strongly Ulam stable. P Rolli in \cite{R09} gave a surprisingly simple example showing that $F_2$ is not strongly Ulam stable. Leveraging the latter,  Burger, Ozawa and Thom proved in \cite{BOT13} that groups containing non-abelian free subgroups are not strongly Ulam stable. 
We further this line of inquiry by showing that any wreath product of an abelian group with a non-amenable group is not strongly Ulam stable. 


\begin{mytheor}\label{thm: ulam non-stability}
A wreath product $A \wr G$ of a non-trivial abelian group $A$ and a non-amenable group $G$ is not strongly Ulam stable.
\end{mytheor}

It is an open problem whether all non-amenable groups are not strongly Ulam stable.

A unitary representation $\rho: G \to \Uni(G)$ is called stronly rigid if for every $\varepsilon>0$ there is $\delta>0$ such that any unitary representatio:n $\rho'$ with distance at most $\delta$ from $\rho$ is conjugated by a unitary operator $U$ such that $\lVert U - \Id_{\Hi}\rVert < \varepsilon$. 
In \cite{BOT13} it was shown that groups containing a non-abelian free subgoup have non strongly rigid representations. 
We again show that this result could be extended to wreath prodcuts of non-trivial and non-amenable groups. 

\begin{mytheor}\label{thm: not strong rigidity}
A countable group containing as a subgroup a wreath product of a non-trivial and a non-amenable group has a not strongly rigid representation
\end{mytheor}

A countable Borel measure preserving equivalence relation $E$ on a standard probability space $(X,\mu)$ is an equivalence relation which is a Borel subset of $X \times X$, such that every equivalence class is at most countable and such that any partial Borel injection $\psi$, whose graph is a subset of $E$, is measure preserving. All equivalence relations on standard probability spaces in the sequel would be such.

The {\em full group} $\llbracket E\rrbracket$ of an equivalence relation $E$ is the group of all measure preserving automorphisms whose graphs are subsets of $E$. The {\em uniform distance} between two elements $T_1, T_2$ of $\llbracket E\rrbracket$ is defined as the measure of the set where $T_1$ and $T_2$ differ; the topology defined by this metric is called the {\em uniform topology}.

We would say that a countable Borel measure preserving equivalence relation $E$ on a standard probability space $(X,\mu)$ orbitally contains a group $N$ if there is an essentially free action of $N$ on $(X,\mu)$ such that the $N$-orbit of almost every point is contained in the $E$-equivalence class of that point. 

\begin{mytheor}\label{thm: dense subgroup}
	Assume that $E$ is a countable Borel measure preserving equivalence relation that orbitally contains a countable non-unitarizable group. Let $G$ be a countable dense subgroup of the full group $\llbracket E \rrbracket$ of $E$. We have the following:
	\begin{itemize}
		\item if $E$ orbitally contains a non-unitarizable group, then $G$ is non-unitarizable;
		\item if $E$ orbitally contains a group that has a not strongly rigid representation, then $G$ has a not strongly rigid representation. 
	\end{itemize}
\end{mytheor}

There are many examples of equivalence relations satisfying the requirements of the theorem above. First, any orbit equivalnece relation of a Bernoulli action by Theorem \ref{thm: Bernoulli has Day property}. Second, in \cite{B19} it is mentioned that the methods of that paper could be combined with those of \cite{BHI18} to show that any Bernoulli extension of a non-hyperfinite ergodic equivalence relation orbitally contains the free group $F_2$. I refer the reader to the aforementioned papers for the details. 

{\em Acknowledgements.} I thank Lewis Bowen for pointing to his work on orbit inclusion of free groups in non-amenable ones and for his remarks and suggestions. I'm grateful to Danil Akhtiamov for discussions of stability and rigidity properties. The research was started during my postdoc in the Einstein Institute of Mathematics at The Hebrew University supported by the Israel Science Foundation grant 1702/17, and carried on in the Euler Mathematical Institute at St. Petersburg State University, supported by Ministry of Science and Higher Education of the Russian Federation, agreement no. 075--15--2022--287.

\section{Fibered Hilbert spaces and representations}

For a Hilbert space $\Hi$ we denote $\Aut(\Hi)$ the group of all bounded invertible operators on $\Hi$. A representation $\pi$ of a countable group $G$ on a Hilbert space $\Hi$ is a group homomorphism from $G$ to $\Aut(\Hi)$. We say that two representations $\pi$, $\pi'$ of the same group $G$ on spaces $\Hi$ and $\Hi'$ are {\em similar} if there is a bounded invertible operator $Q$ from $\Hi$ to $\Hi'$  such that $Q \pi(g) Q^{-1} = \pi'(g)$, for all $g \in G$. We would say also that this similarity is given by operator $Q$.

Let $(X, \mu)$ be a standard probability space and $\Hi$ be a separable Hilbert space. A {\em fibered Hilbert space} $L_2(X,\mu, \Hi)$ over $(X,\mu)$ with fiber $\Hi$ is the space of all classes of almost everywhere equal measurable functions $f$ from $X$ to $\Hi$ such that $\int_X \lVert f(x) \rVert^2 d\mu(x) < +\infty$. The scalar product is given by $\langle f, h\rangle = \int_X \langle f(x), h(x)\rangle d\mu(x)$. It is not hard to check that this space endowed with the aforementioned scalar product is indeed a Hilbert space. Note that $L_2(X, \mu, \C)$ is the standard $L_2$ over $(X,\mu)$, and that $L_2(X,\mu, \Hi)$ is canonically isomorphic to the Hilbert space tensor product $L_2(X,\mu) \otimes \Hi$. Let $L_2(X,\mu, \Hi)$ be a fibered Hilbert space. For every measurable subset $C$ of $X$ we define a projector operator $P_C$ by 
$(P_C(f))(x) = 0$ if $x \notin C$, and $(P_C(f))(x) = f(x)$ if $x \in C$. Note that $P_{C_1} \circ P_{C_2} = P_{C_2} \circ P_{C_1} = P_{C_1}$ if $C_1 \subset C_2 \subset X$. If $(C_i)$ is a sequence of measurable subsets of $X$ such that $\mu(C_i \Delta D) \to 0$ for some measurable subset $D$ of $X$, then $P_{C_i}$ converge to $P_D$ in the strong operator topology. In fact there is a natural embedding $\iota : L_\infty(X,\mu) \hookrightarrow \bop(L_2(X, \mu, \Hi))$ given by $(\psi f) (x) = \psi(x) f(x)$ for $\psi \in L_\infty(X, \mu)$, $f \in L_2(X,\mu, \Hi)$ and a.e. $x \in X$. We will sometimes identify $L_\infty(X, \mu)$ with its image under the aforementioned embedding. For convenience we 

We will call an operator $B$ on $L_2(X,\mu,\Hi)$ {\em totally fibered} if $B P_C = P_C B$ for all measurable subsets $C$ of $X$. So an operator is totally fibered iff it lies in the commutant of $L_\infty(X,\mu)$.

The following proposition gives an equivalent definition:
\begin{prop}\label{prop: totally fibered operator}
An operator $B$ on $L_2(X,\mu, \Hi)$ is totally fibered iff there is a measurable field of operators $(B_x)_{x\in X}$ on $\Hi$ (i.e. the map $x \mapsto B_x$ is a measurable map from $X$ to $\bop(\Hi)$), such that the norms of $B_x$ are essentially bounded, and that $(B(f))(x) = B_x(f(x))$, for all $f \in L_2(X, \mu, \Hi)$, and a.e. $x \in X$. Moreover, $\lVert B \rVert = \esssup_{x \in X} \lVert B_x\rVert$.
\end{prop}

We will call the mentioned field of operators $(B_x)$ the {\em explicit fibered form} of the totally fibered operator.

\begin{proof}[Sketch of proof]
It is straightforward that the second definition implies the first one. For the other direction, we may take a countable dense subset $V$ of $\Hi$ that is a vector space over the field of complex-rational numbers $\Q + \Q i$.
Notice that the map $C \mapsto \langle B P_C (1 \otimes v), 1 \otimes u\rangle$, for all Borel subsets $C$ of $X$, is a finite measure on $X$, absolutely continuous with respect to $\mu$, for any $v,u \in \Hi$, and the absolute value of its Radon-Nikodym derrivative is bounded by $\lVert B \rVert \lVert u\rVert \lVert v\rVert$. Using these derrivatives, we refine the quadratic forms corresponding to $B_x$ (it is enough to carry out the procedure for all $v,u \in V$). This way we get the field of operators $(B_x)$, and it is easy to check that their norm should be essentially bounded by $\lVert B\rVert$. 
\end{proof}

It is not hard to see that $\lVert B\rVert = \esssup_{x} \lVert B_x\rVert$. We also note that two totally fibered operators are equal whenever their explicit fibered forms are equal almost surely. We note also that an operator is fibered iff it commutes with a weakly dense subsset of $L_\infty(X, \mu)$.  the following is easy to check.
\begin{lem}\label{lem: fibered unitary}
Let $B$ be a fibered operator and $(C_x)$ is its explicitly fibered form. Operator $B$ is unitary iff $(C_x)$ is unitary for almost every $x$.
\end{lem}


Note that for each measure preserving automorphism $T$ of $(X, \mu)$ there is a corresponding generalized Koopman operator $U_T$:
$(U_T(f))(x) = f(T^{-1}(x))$, for any $f \in L_2(X,\mu, \Hi)$. Note that $U_T^{-1} P_C U_T = P_{TC}$ for any measurable subset $C$ of $X$.
An operator $Q$ on $L_2(X,\mu, \Hi)$ is called a {\em skew-fibered operator} if it is a product of a generalized Koopman operator and a totally fibered operator. Note that this decomposition is unique iff the totally fibered operator is nonzero almost everywhere. In particular, the decomposition is unique for  invertible fiberd operators. It follows from our discussion of totally fibered operators that there is another equivalent definition:

\begin{prop}\label{prop: skew-fibered}
	Let $Q$ be an operator on $L_2(X, \mu, \Hi)$ and $T$ be a  measure preserving transformation of $(X,\mu)$.
	The following are equivalent:
	\begin{enumerate}
		\item $Q$ is skew-fibered over $T$;
		\item $QP_C = P_{T(C)} Q$, for all measurable $C \subset X$;
		\item $Q \psi$ = $U_T(\psi) Q$ for every $\psi \in L_\infty(X, \mu)$ .
		\item $Q \psi$ = $U_T(\psi) Q$ for a weakly spanning set of $\psi \in L_\infty(X, \mu)$ .
	\end{enumerate}
	Here $U_T$ is the Koopman operator associated with $T$. 
\end{prop}

\begin{proof}
The second, third and fourth statements are trivially equivalent. 
 $(1)\Rightarrow(2)$ direction is straightforward, for the reverse we notice that $(U_T)^{-1}Q$ is a totally fibered operator.
\end{proof}

We note that two skew-fibered operators are equal whenever the underlying measure preserving transformations are equal and corresponding totally fibered operators are equal fiberwise. 

Let us fix an essentially free measure preserving action of a countable group $G$ on a standard probability space $(X, \mu)$. We would use the notation $gx$ for the result of the action of element $g$ of group $G$ on element $x$ of space $X$. A fibered representation with that action as a base and Hilbert space $\Hi$ as a fiber is a measurable map $\bp: G \times X \to \Aut(\Hi)$ such that
\begin{enumerate}
\item $\bp(g_1, g_2 x) \circ \bp(g_2, x) = \bp(g_1 g_2, x)$, for all $g_1, g_2 \in G$ and $x \in X$;
\item $\bp(1_G, x) = \Id$, for all $x \in X$;
\item there is a constant $K$ such that $\rvert \bp(g,x) \lvert < K$ for all $g \in G$ and almost every $x \in X$.
\end{enumerate}
We would also call the latter object a representation of the action $G \curvearrowright (X,\mu)$. If we waive the first requirement, the resulting object will be called a {\em quasireprersentation of the action}. We define the {\em defect} of a quasirepresentation $\df(\bp)$ in the following way:
\[
\df(\bp) = \esssup_{x \in X} \sup_{g_1,g_2 \in G} \lVert \bp(g_1 g_2 , x) - \bp(g1, g_2 x) \circ \bp(x,g_2) \rVert.
\]
We will sometimes say that $\bp$ is an {\em $\varepsilon$ - represenation} of action $G \curvearrowright (X,\mu)$ if it is a quasirepresentation with defect smaller than $\varepsilon$.

Note that fibered representation $\bp$ gives rise to a usual representation $\pi$ on the fibered space $L_2(X,\mu, \Hi)$: $(\pi(g)(f))(gx) = \pi(g,x)(f(x))$, for all $g \in G$, $f \in L_2(X, \mu, \Hi)$ and $x \in X$. 
A simplest example of a fibered representation is given by the tensor product of any representation $\tau: G \to \Aut(\Hi)$ and a left Koopman representation given by a measure preserving action of $G$ on $(X, \mu)$. We will denote that fibered representation $(G \curvearrowright X) \otimes \tau $.
The following lemma follows immediately from the properties of skew-fibered operators.

\begin{prop}
Let $\pi$ be a quasirepresentation of a countable group $G$ on a fibered Hilbert space $L_2(X,\mu, \Hi)$ and let $G \curvearrowright (X,\mu)$ be a fixed measure preserving acition.
The following are equivalent:
\begin{enumerate}
\item representation $\pi$ corresponds to a representation $\bp$ of the action $G \curvearrowright (X,\mu)$  on a Hilbert space $\Hi$;
\item for every $g \in G$, the operator $\pi(g)$ is skew-fibered over the measure preserving map $x \mapsto gx$;
\item $\pi(g) P_C = P_{gC} \pi(g)$, for all $g \in G$ and every measurable subset $C$ of $X$;
\item $\pi(g) \psi = U_g(\psi) \pi(g)$, for every $g \in G$ and $\psi \in L_\infty(X,\mu)$;
\item $\pi(g) \psi = U_g(\psi) \pi(g)$, for every $g \in G$ and for a weakly spanning set of $\psi \in L_\infty(X,\mu)$;
\end{enumerate}
\end{prop}

\begin{proof}
Trivially follows from proposition \ref{prop: skew-fibered}.
\end{proof}

Let actions of countable groups $G$ and $N$ on $(X,\mu)$ constitute a measurable inclusion of $N$ into $G$; we remind that this means that the $N$-orbit of almost every point in $X$ is a subset of its $G$-orbit.

If $\bp$ is a fibered representation of $G$, then there is a naturally obtained {\em restricted} fibered representation of $N$ (we choose this name because the construction could be considered an analog of restriction of representation construction: a representation of a group could be restricted to a representation of its subgroup).
Indeed, for each element $h$ of $N$ there is a partition $X = \bigsqcup_{g \in G} C_{g,h}$ such that $hx = gx$ for each $g \in G$ and $x \in C_{g,h}$.
The representation is given by the formula $\bar{\rho'}( h,x) = \bp(g,x)$, for $x \in C_{g,h}$. It is easy to check that this representation is uniformly bounded (it is enough to check boundedness fiberwise). 

\begin{lem}\label{lem: fibered unitariser induction}
Suppose that there are actions of groups $G$ and $N$ on a standard probability space $(X,\mu)$ forming an orbit inclusion of $N$ into $G$, such that the $G$-action is ergodic. Suppose that $\bp$ is a fibered over the $G$-action representation of $G$, and $\bar{\rho'}$ is the restriction of that representation into $N$-representation. Suppose that $\pi$ is unitarizable with unitarizing operator $B$ such that $B$ is strongly fibered over $(X,\mu)$. Then $\rho'$ is also unitarizable with the same unitarizing operator.
\end{lem}

\begin{proof}

We want to prove that $(\rho'(h))^* B \rho(h) = B$, for all $h \in N$. Note that a conjugate of a skew-fibered operator is a skew-fibered operator. It easy to check now that underlying transformations are the same and that the both sides are equal fiberwise.
\end{proof}

\section{Representations of wreath products}\label{sec: representations of wreath}

Let $A$ and $G$ be groups. Let $N = \bigoplus_G A$ be a direct sum of $G$-indexed copies of $A$.  Let $a^g$ denote the copy of $a \in A$ in the ``coordinte'' $g$ of the direct sum.  The wreath product is  free product $G   * N$ factorized by relations  $t^{-1}(a^g)t = a ^{gt}$ for $a \in A$ and $g,t \in G$.

Let $(X,\mu)$ be a standard probability space and let $\Hi$ be a Hilbert space. We note that there is a natural embedinging of $C^*$-algebras $\iota: L_\infty(X,\mu) \hookrightarrow \bop (L_2(X, \mu, \Hi))$.

Let $A$ be a finite or countable abelian group. Let $(Y, \nu)$ be its Pontryagin dual endowed with the Haar measure. Let $G$ be a countable group. We have the Bernoulli shift-action of $G$ on the space $(X, \mu) = Y^G$ (endowed with the product measure), namely  $(g(t))(h) = t(hg)$, for any $g, h \in G$ and $t \in Y^G$. Note that $Y^G$ is the Potryagin dual of $\bigoplus_G A$, so we may identify the latter with a subset of $L_\infty(X, \mu)$.  Naturally, we obtain a representation of $N$ on the Hilbert space $L_2(X, \mu, \Hi)$. 
For $a \in A$ and $g \in G$ let $a^g$ be the copy the instance of $a$ in the $g$'th copy of $A$.  We note that $\iota(a^g) = U_{g}^{-1} \iota(a^e) U_g$ and $U_g^{-1} \iota(a^h) U_g = \iota(a^{hg})$.

\begin{prop}\label{prop: fibered to wreath}
Assume that $\pi$ is a quasirepresentation of $G$ that is fibered over the Bernoulli action $G \curvearrowright(X, \mu)$ ($X = Y^G$, $Y = \hat{A}$).  
There is a natural extension of $\pi$ into a qusirepresentation  $\pi'$ of $A \wr G$ such that the following holds:

\begin{enumerate}
\item $\pi' \vert_{\bigoplus_G A}$ coincides with $\iota \vert_{\bigoplus_G A}$;
\item $\df(\pi') = \df(\pi)$;
\item $\pi'$ is unitary whenever $\pi$ is unitary; 
\item $\pi'$ is uniformly bounded whenever $\pi$ is (with the same constant); 
\item If $\pi_1$, $\pi_2$ are two representations of $G$ fibered over $G \curvearrowright (X, \mu)$, then $\lVert \pi'_1 -  \pi'_2\rVert = \lVert \pi_1  - \pi_2\rVert $, where $\pi'_1$, $\pi'_2$ are corresponding extensions to $A \wr G$. 
\end{enumerate}

\end{prop}

\begin{proof}

Each $\gamma \in A \wr G$ could be presented uniquely as a product $gh$, where $g \in G$ and $h \in N$. We define $\pi'(gh) = \pi(g) \iota(h)$. 
The only slightly nontrivial claim is the one about the defect. It follows from the following chain of expressions ($g_1, g_2 \in G$, $h_1, h_2 \in N$):
\begin{eqnarray*}
\pi'(g_1 g_2 h_1^{g_2} h_2) \\
\pi'(g_1 g_2) (\pi'(g_2))^{-1} \pi'(h_1) \pi'(g_2) \pi'(h_2) \\
\pi'(g_1) \pi'(g_2) (\pi'(g_2))^{-1} \pi'(h_1) \pi'(g_2) \pi'(h_2) \\
\pi'(g_1) \pi'(h_1) \pi'(g_2) \pi'(h_2) \\
\pi'(g_1 h_1) \pi'(g_2 h_2).
\end{eqnarray*}
All pairs of subsequent expressions are equal, except for the second and the third, where the distance is at most the defect of $\pi$.

\end{proof}


\begin{prop}\label{prop: wreath to fibered}
Let $G \curvearrowright (X, \mu)$ (where $X = Y^G$, $Y = \hat{A}$) be a Bernoulli action. Assume that $\bp'$ is a quasirepresentation of $A \wr G$ on $L_2(X, \mu, \Hi)$ such that $\pi' \vert_N = \iota$. Then $\pi = \pi' \vert_G$ is a fibered representation over the $G$-action with fiber $\Hi$.
\end{prop}

\begin{proof}
Follows trivially from proposition \ref{prop: skew-fibered}
\end{proof}

\begin{prop}
Let $G$ be a countable group and let $G \curvearrowright (X,\mu)$ be a measure preserving action. Let $C$ be a measurable subset of $X$ such that $\{C , X \setminus C \}$ is a generating partition for the action $G \curvearrowright (X,\ mu)$.
Let $\xi \in \T$, $\xi \neq 1$. If $\pi$ is a quasirepresenation of $G$ on $L_2(X, \mu, \Hi)$ such that $(\pi(g))^{-1} T_h \pi(g) = T_{hg}$ for all $g,h \in G$, where $T_g = \xi P_{gC} + P_{X \setminus gC}$. When $\pi$ is fibered over the action $G \curvearrowright (X, \mu)$ with fiber $\Hi$.
\end{prop}

\section{Representations of equivalence relations}\label{sec: reprpresentations of equiv}

Let $E$ be a Borel equivalence relation on a standard probability space $(X,\mu)$ with (at most) countable equivalence classes. We will say that $E$ is {\em measure preserving} if every partial Borel injection $\psi$ whose graph is a subset of $E$ is measure preserving. In the sequel all equivalence relations on standard probability spaces would be Borel with at most countable classes and measure preserving unless otherwise explicitly stated.
The space $E$ could be endowed naturally with a $\sigma$-finite measure: 
\begin{equation*}
\mu^{E}(A) = \int_X \lvert A \cap \{ x\} \times X\rvert d\mu(x) =  \int_X \lvert A \cap  X \times \{ x\}\rvert d\mu(x), \, \text{for } A \subset X.
\end{equation*}
So it makes sense to talk about almost every pair $(x,y) \in E$. 
Denote $\prescript{}{3}{E}$ the set of all triples $(x,y,z)$ such the $(x,y), (y,z) \in E$. In a similar way, we may endow $\prescript{}{3}{E}$ with a $\sigma$-finite measure, so we may talk about almost all triples $(x,y,z) \in \prescript{}{3}{E}$. Note that if $A$ is a conull subset of $E$, then there is a conull subset  $X'$ of $X$ such that $(X' \times X') \cap E \subset A$. Similarly, if $B$ is a conull subset of $\prescript{}{3}{E}$, then there is a conull subset $X'$ of $X$ such that $\prescript{}{3}{E} \cap {X'}^3 \subset B$.

Let $\Hi$ be a Hilbert space. A Borel map $\eta : E \to \Aut(\Hi)$ is called a {\em representation of equivalence realtion} $E$ if there is a conull subset $X'$ of $X$ and a constant $K$ such that 
\begin{enumerate}
\item $\eta(x,x) = \Id$, for almost every $x \in X$;
\item $\lvert \eta(x,y) \rvert \leq K$ for almost every $(x,y) \in E$; 
\item $\eta(y,z) \circ \eta(x, y) = \eta(x,z)$, for almost every $(x,y), (y,z) \in E$.
\end{enumerate}

If we drop the last requirement, we will call the resulting object a {\em quasirepresentation}. 

The defect $\df(\eta)$ of a quasirepresentation $\eta$ is defined as 
\begin{equation*}
\esssup_{(x,y,z) \in \prescript{}{3}{E}} \lVert \eta(y,z) \circ \eta(x, y) - \eta(x,z) \rVert.
\end{equation*}
If $\eta$ is a quasirepresentation and $\df(\eta) \leq \varepsilon$, then we will also call $\eta$ an {\em $\varepsilon$-representation}.

For an equivalence relation $E$ on a standard probability space $(X, \mu)$ and a Hilbert space $\Hi$ we will denote $\Rep(E, \Aut(\Hi))$ the space of all representations, and $\QRep(E, \Aut(\Hi))$ --- the space of all quasirepresentations. We define the norm $\lVert \eta \rVert$ of  quasirepresentation $\eta$ as 
\begin{equation*}
\esssup_{(x,y) \in E} \lVert \eta(x,y) \rVert,
\end{equation*}
and the distance $\lVert \eta_1 - \eta_2\rVert$ between two quasirepresentations as 
\begin{equation*}
\esssup_{(x,y) \in E} \lVert \eta_1(x,y) - \eta_2(x,y) \rVert.
\end{equation*}

It is not hard to see that a fibered representation of a group over its essentially free measure preserving action is the same thing as a representation of the orbit equivalence relation of that action. Indeed we may define 
$\eta(x,y) = \bp(g,x)$, where $x,y \in X$ are in the same $G$-orbit, and $g$ is such unique element of $G$ that $gx = y$, where $\bp$ is a fibered representaion of group $G$ over the essentially free action. The same holds with regards to fibered quasirepresentations.

Equivalence relation $E$ is ergodic if any $E$-closed measurable subset of $X$ has either zero or full measure.
If $E_1 \subset E_2$ are two equivalence relations on a standard probability space $(X,\mu)$ and $E_2$ is ergodic, it makes sense to define the {\em index} $[E_2:E_1]_{(X,\mu)}$ as the number of $E_1$-classes into which the $E_2$-class of $\mu$-almost every point splits.

Given a quasirepresentation $\eta$ of an equivalence relation $E_1$ on $(X,\mu)$ with a fiber $\Hi$, we can induce a quasirepresentation $\eta^{E_2}$ of $E_2$ with fiber $\Hi^{[E_2 : E_1]_{(X,\mu)}}$
(this construction could be considered an analog of the induced representation construction, where from a representation of a subgroup we obtain a repressentation of the containing group).
We need some notation to do this. Let $I$ be the segment $[0, [E_2:E_1]_{(X,\mu)})$ (finite or infinite) of natural numbers.
The fibers $(\lbrace x\rbrace \times X) \cap E_2$, for $x \in X$, split into $E_1$-equivalence classes. We want to pick a measurable system of representatives for these equivalence classes.
To do so, in few moments we will define two measurable functions: $\rp : E_2 \to I$ and $\en : X \times I \to X$ that will satisfy the following properties:

\begin{enumerate}
\item $(\en(x,i), x) \in E_2$, for all $i \in I$ and a.e. $x \in X$;
\item $(\en(x,\rp(x,y)), y) \in E_1$, for a.e. $x \in X$ and every $y \in X$ such that $(x,y) \in E_2$;
\item $\rp(x, \en(x,i)) = i$, for a.e. $x \in X$ and for all $i \in I$.
\end{enumerate}
Given these functions, we are ready to define an induced quasirepresentation $\eta^{E_2}$ of $E_2$ with fiber $\Hi^I$ from representation $\eta$ of $E_1$ with fiber $\Hi$. 

For any $f \in \Hi^I$ and $(x,y) \in E_2$ we define 
\[ 
\left((\eta^{E_2}(x,y))(f)\right)(i) = \eta(\en (x, \rp(x, \en(y,i))), \en(y,i))(\rp(x,\en(y,i))). 
\]
Let us give a more intuitive description. Take function $f$, we imagine that it lies in the fiber over $x$. It has ``coordinates'' $f(j)$ for all $j \in i$. We ``hook'' each such coordinate to the point $(x, \en(x,j))$ of $E_2$. So each coordinate is hooked to a representative of an $E_1$-equivalence class in the intersection $(\{x\} \times X) \cap E_2$. To apply $\eta^{E_2}(x,y)$ to $f$, we look for the representative of the $E_1$-class containing point $(y, \en(x,j))$ in $(\{y\} \times X) \cap E_2$. This representative will be $\en(y,\rp(y, \en(x,j)))$ and have index $\rp(y, \en(x,j)) \in I$.  Now, the change of the index account for the permutation of the coordinates. To handle the change of the representative, we should apply $\eta(\en(x,j), \en(y,\rp(y, \en(x,j))))$ to $f(j)$, and obtain $\rp(y, \en(x,j))$-th coordinate of $(\eta^{E_2}(x,y))(f)$.

Now let us show how to construct functions $\rp$ and $\en$. By the Lusin-Novikov uniformization theorem (see \cite[Theorem 18.10]{Ke95}), there is a sequence of measurable functions $(\psi_i)$ from $X$ to $X$, union of whose graphs is equivalence relation $E_2$. We also require that $\psi_0$ is the identity function. Now, we call a pair $(x,i)$ a new pair if $(\psi_j(x), \psi_i(x)) \notin E_1$, for all $j < i$. We define $\en(x,j)$ to be the $\psi_i(x)$, where $i$ is such that $(x,i)$ is the $j$'th new pair in the sequence $(x,0), (x,1), (x,2) \ldots$.  It is not hard to see that $\en$ is a measurable a.e. defined function. We define $\rp(x,y)$ to be the unique $i$ such that $(\en(x,i), y) \in E_1$.

Note that $\Hi_0$, the subspace of $\Hi^I$ that corresponds to the $0$'th coordinate, is $E_1$-invariant, i.e. there is a full-measure subset $X''$ of $X$ such that $\eta^{E_2}(x,y)(v) \in \Hi_0$, for any $v \in \Hi_0$ and $x,y \in X''$ such that $(x,y) \in E_1$, in fact, $(\eta^{E_2}(x,y)(v))_0 = \eta(x,y)(v_0)$, where $v = (v_0, v_1, \ldots)$ and $v_i = 0$, for $i>0$. We might notice also that the restriction of $\eta^{E_2}$ into a representation of $E_1$ contains $\eta$ as a subrepresentation. 

Another important obseravtion is that the induced representation construction $\eta^{E_2}$ is not unique, but any other induced representation ${\eta^{E_2}}'$ is uniformly equivalent: there is a measurable field of invertible operators $(A_x)_{x \in X}$ such that their norms and norms of their inverses are uniformly bounded almost surely, and $A_y \eta^{E_2}(x,y) A_x^{-1} = {\eta^{E_2}}'(x,y)$, for some conull subset $X'$ of $X$ and every $x,y \in X'$, $(x,y) \in E_2$.

Let us some up the properties of the induction construction. These are easy to check.
\begin{prop}
Let $E_1 \subset E_2$ be measure-preserving equivalence realtions on a standard probability space $(X,\mu)$. Let $\Hi$ be a Hilbert space. Assume that $E_2$ is ergodic. Let us identify $\Hi$ and the zeroth summand of $\Hi^{[E_2 : E_1]}$.
There is a map from $\QRep(E_1,\Aut(\Hi))$ to $\QRep(E_2, \Aut(\Hi^{[E_2:E_1]}))$, $\eta \mapsto \eta^{E_2}$ that has the following properties ($\eta$, $\eta_1$, $\eta_2$ are quasirepresrntions of $E_1$):
\begin{enumerate}
\item $(\eta^{E_2}(x,y)) (v) = (\eta(x,y))(v)$ for almost all $(x,y) \in E_1$ and all $v \in \Hi$;
\item if $\eta$ is unitary then $\eta^{E_2}$ is unitary;
\item $\lVert \eta^{E_2}\rVert = \lVert \eta \rVert$;
\item $\lVert \eta_1^{E_2} - \eta_2^{E_2}\rVert = \lVert \eta_1 - \eta_2\rVert$;
\item $\df(\eta^{E_2}) = \df(E_2)$.
\end{enumerate}
\end{prop}

\section{Lamplighters over non-amenable groups are non-unitarizable}

Let $E$ be an equivalence relation on a standard probability space $(X,\mu)$. We would say that $E$ is {\em unitarizable} if for every representation $\eta$ of $E$, there is a essentially uniformly bounded field of invertible operators $B_x$ such that their inverses form an essentially uniformly bounded field as well, such that $$\eta^*(x,y) B_y \eta(x,y) = B_y$$ for almost every $y \in X$. We will call $(B_x)$ a {\em unitarizing field} for $\eta$.

\begin{lem}
If $E$ orbitally contains a non-unitarizable group $N$, then $E$ is non-unitarizable.
\end{lem}

\begin{proof}
Assume that $E$ is unitarizable, Let $\tau$ be any uniformly-bounder representation of $N$ on a Hilbert space $\Hi$. Consider the action of $N$ on $(X, \mu)$ given by the orbit inclusion of $N$ into $E$. Denote $E_N$ the orbit equivalence relation of that action. We may construct the tensor-product representation $\bp$ of the action $N \curvearrowright (X, \mu)$ and from that a representation $\rho$ of the equivalence relation $E_N$. Now let $\rho^{[E : E_N]}$ be the induced representation of $E$ on the space $\Hi^{[E:E_N]}$.  Let $(B_x)$ be the unitarizing operator for $\rho^{[E : E_N]}$.Note that the copy of $\Hi$ corresponding to the 0'th coordinate in $\Hi^{[E: E_N]}$ is $\rho^{[E : E_N]}(x,y)$-invariant for a.e. $(x,y) \in E_N$, we denote $P_\Hi$ the projector from $\Hi^{[E:E_N]}$ to said copy  It is not hard to check that $P_\Hi B_x P_\Hi$ forms a unitarizing field for $\rho$. Now it is easy to check that $B$ is a unitarizing operator for the representation $\pi$ of $N$ corresponding to the representation $\bp$ of $N \curvearrowright (X, \mu)$. We see that $\pi$ is unitarizable, but it contains $\tau$ as a subrepresentation, so $\tau$ is unitarizable as well.
\end{proof}

\begin{proof}[Proof of Theorem \ref{th: main}]
Unitarizability passes to subgroups, so it is enough to show that $A \wr G$ is non-unitarizable for any countable or finite abelian group $A$ and any countable non-amenable group $G$. Consider the Bernoulli action $G \curvearrowright (X,\mu)$ ($X = Y^G$, $Y = \hat{N}$, $N = \bigoplus_G A$) from Section \ref{sec: representations of wreath}.  Assume that $A \wr G$ is unitarizable. We will prove that $E$ is unitarizable as well. Let $\eta$ be any representation of $E$ on a Hilbert space $\Hi$. From this we get a uniformly-bounded representation $\pi$ of $G$ on $L_2(X, \mu, \Hi)$. We extend the latter to representation $\pi'$ of $A \wr G$ as described in Section \ref{sec: representations of wreath}. trivially, $\pi'$ is uniformly bounded, so it is unitarizable. Let $B$ be the unitarizing operator for $\pi'$. By definition of the unitarizing operator, we may see that $B$ commutes with  $\pi'(h)$ for all $h \in N = \bigoplus_G A$ , so $B$ is totally fibered. It is then not hard to check that its explicit fibered form $(B_x)$ is a unitarizing field for $\eta$.
\end{proof}

\section{Lamplighters have non strongly rigid representations}

Let $G$ be a countable group. We say that it is {\em deformation rigid} if for every $\varepsilon > 0$ there is $\delta > 0$ such that any two unitary representations $\rho_1, \rho_2$ on the same Hilbert space such that $\lVert \rho_1 - \rho_2 \rVert < \delta$ are congugated by a unitary operator $U$ satisfying  $\lVert \Id - U\rVert < \varepsilon$.

A unitary representation $\rho$ of a countable group $G$ on a Hilbert space $\Hi$ is called {\em strongly rigid} if for every $\varepsilon > 0$ there is $\delta > 0$ such that for any unitary representation $\rho_1$ of $G$ on $\Hi$ such that $\lVert \rho - \rho_1 \rVert < \delta$ there is unitary operator $U$ satisfying $\lVert \Id - U \rVert < \varepsilon$ that conjugates the two represetations. Note that it is a priori weaker condition on a group that all of its representaions are strongly rigid, than that of deformation rigidity of the group, since the latter requires uniformity.

A unitary representation $\rho$ of a measure preserving action $G \curvearrowright (X, \mu)$ on a Hilbert space $\Hi$ is called strongly rigid if for every $\varepsilon > 0$ there is $\delta > 0$ such that every unitary representations $\rho'$ of that action on Hilbert space $\Hi$ satisfying $\lVert \rho - \rho' \rVert < \delta$, is conjugated to $\rho$ by a unitary operator $U$ such that $\lVert \Id - U\rVert < \varepsilon$.

Let $E$ be an equivalence relation on a standard probability space $(X,\mu)$. We will say that two representations $\eta_1$ and $\eta_2$ of $E$ are {\em conjugated} by a measurable field of operators $(U_x)_{x \in X}$ if for a.e. $(x,y) \in E$ we have $\eta_1(x,y) U_x = U_y \eta_2(x,y)$.
A unitary representation $\eta$ on a Hilbert space $\Hi$ of an equivalence relation $E$ on a standard probability space $(X,\mu)$ is called {\em strongly rigid} if for every $\varepsilon > 0$ there is $\delta > 0$ such that every unitary representation $\eta_1$ of $E$ on $\Hi$ satisfying $\lVert \eta - \eta_1\rVert > 0$ is conjugated to $\eta$ by a measurable field of unitary operators $(U_x)_{x \in X}$ such that $\lVert \Id - U_x \rVert < \varepsilon$ for a.e. $x \in X$. 


\begin{prop}
Let $E_1 \subset E_2$ be two equivalence relations on $(X,\mu)$. If all unitary representations of $E_2$ are strongly rigid, then all unitary representations of $E_1$ are strongly rigid.
\end{prop}

The latter proposition trivially follows form the induced representation construction together with the following lemma that incapsulates the argument from Proposition 4.4.2 and Lemma 4.5 of \cite{BOT13} .

\begin{lem}
Let $E$ be an equivalence relation on a standard probability space $(X,\mu)$. For every $\varepsilon > 0$ there is a $\delta > 0$ such that the following holds. Let $\Hi \subset \Hi'$ be two Hilbert spaces. Let $\eta_1, \eta_2$ be two unitary representations of $E$ on $\Hi'$ such that $\Hi$ is an invariant subspace for both and that they are conjugated by a measurable field of unitary operators $(U_x)_{x \in X}$ satisfying $\lVert U_x - \Id\rVert < \delta$. Then restrictions of $\eta_1 \vert_{\Hi}$ and $\eta_2 \vert_{\Hi}$ are conjugated by the unitary field of unitary operators $(V_x)_{x \in X}$, where $V_x$ is the ``sign'' part of the polar decomposition of $PU_xP$, and $P$ is the projector from $\Hi'$ onto $\Hi$. In addition, $\lVert V_x - \Id_{\Hi}\rVert < \varepsilon$ for a.e. $x \in X$.
\end{lem}

\begin{proof}
Let $W_x = \big\lvert P U_x P\vert_\Hi \big\rvert$. Observe that for a.e. $x,y$ we have $\eta_1(x,y) P U_x P= P U_y P \eta_2(x,y)$ and $\eta_2(x,y) P U^*_x P = P U^*_yP \eta_1(x,y)$. So $\eta_2(x,y) W_x = W_x \eta_2(x,y)$. Using the fact that $P U_x P \vert_\Hi = V_x W_x$, we get that the field of operator $(V_x)$ conjugates $\eta_1 \vert_\Hi$ and $\eta_2\vert_\Hi$ as soon as $\delta$ is small enough (so that $W_x$ is invertible). Since in $V_x$ is obtained by a norm-continuous in the neighbourhood of $Id_{\Hi'}$ map those value on $\Id_{\Hi'}$ is $\Id_{\Hi}$, we obtain the desired.
\end{proof}

\begin{cor}\label{cor: rigidity induction for equivalence}
If $E_1 \subset E_2$ are two equivalence relations on a standard probability space, and $E_2$ is ergodic and $E_1$ has a not strongly rigid representation, then $E_2$ has a not strongly rigid representation.
\end{cor}
\begin{proof}
We use the induction construction together with the previous lemma.
\end{proof}

\begin{lem}
Asuume a countable group $H$ has a not strongly rigid representation. 
Let $E_1$ be the orbit equivalence relation of an essentially free measure preserving action $H \curvearrowright (X,\mu)$ of a group $H$ on a standard probability space $(X,\mu)$. There is a non strongly rigid representation of $E_1$. 
\end{lem}

\begin{proof}
Let $\tau$ be a not strongly rigid representation of $H$ on $\Hi$. This means that there is a constant $c>0$ and a sequence of representations $(\tau_i)$ such that $\lVert \tau - \tau_i\rVert \to 0$ and for any unitary operator $U$ on $\Hi$ that conjugates $\tau$ and $\tau_i$ holds $\lVert \Id_\Hi - U\rVert > c$.
Let $\bp$ be the representation of $H \curvearrowright (X,\mu)$  obtained from the tensor product construction $\tau \otimes (H \curvearrowright (X,\mu))$, let $\pi$ be the corresponding representation of $H$ on $L_2(X,\mu, \Hi)$. In a similar way we obtain representations $\bp_i$ of $H \curvearrowright (X,\mu)$ and representations $\pi_i$ of $H$ on $L_2(X, \mu, \Hi)$. We note that $\lVert \pi_i - \pi\rVert \to 0$. 

I claim that there is a $c'>0$ such that for any unitary $U$ that conjugates $\pi$ and $\pi_i$ holds $\lVert \Id_{L_2(X,\mu, \Hi)} - U\rVert > c'$. Assuming the contrary we may refine a subsequence such that there are unitary operators $U_i$ that conjugate $\pi$ and $\pi_i$  such that $U_i \to \Id_{L_2(X,\mu, \Hi)}$. Note that representation $\tau$ is a subrepresentation of $\pi$ and each $\tau_i$ is a subrepresentation of $\pi_i$, moreover the subspaces from the definition of a subrepresentation are the same copy of $\Hi$ given by the natural embedding $\Hi \to L_2(X, \mu, \Hi)$. For each $i$ let $V_i$ be the sign part of the polar decomposition of the restriction $(P_\Hi U_i P_\Hi) \vert_\Hi$. Note that $V_i$ conjugates $\tau$ and $\tau_i$ and that, by Lemma 4.5 from \cite{BOT13}, $\lVert U_i - \Id_\Hi\lVert  \to 0$. A contradiction. 

In the claim above we proved that there is no sequence of unitary operators tending to $\Id$ yet alone a sequence of totally fibered unitary operators (=measurable  fields of unitary operators that would conjugate $\bp$ and $\bp_i$ and hence the correpsonding representations of $E_1$).
\end{proof}

\begin{lem}\label{lem: non-rigid from group to orbitally containing relation}
If $E$ is an ergodic equivalence relation that orbitally contains a group with a non strongly rigid representation, then $E$ has a not strongly rigid representation.
\end{lem}

\begin{proof}
A consequence of the previous lemma and Corolary \ref{cor: rigidity induction for equivalence}.
\end{proof}

\begin{lem}
If $E$ is an orbit equivalence relation of a Bernoulli action of a non-amenable group, then $E$ has a not strongly rigid representation.
\end{lem}

\begin{proof}
It is know that there is a unitary representation of $F_2$ on a Hilbert space $\Hi$ that is not strongly rigid (see \cite{BOT13} Corollary 4.7). 
So the lemma follows from the previous lemma, corolary and the Bowen's theorem \ref{thm: Bernoulli has Day property}.
\end{proof}

\begin{proof}[Proof of Theorem \ref{thm: not strong rigidity}]
Since existence of not strongly rigid representations passes to overgroups (Proposition 4.4.2 from \cite{BOT13}), it is enough to show that if $A$ is a non-trivial countable abelian group and $G$ is a countable non-amenable group, then $A \wr G$ has a not strongly rigid representation. Consider the Bernoulli action $G \curvearrowright (X, \mu)$ ($X = Y^G$, $Y = \hat{A}$)of $G$ from Section \ref{sec: representations of wreath}.  Let $E$ be the equivalence relation of that action. It has a not strongly rigid representation $\eta$ on a Hilbert space $\Hi$.  From this we get a fibered over $G \curvearrowright (X, \mu)$ representation $\pi$ of $G$ on the space $L_2(X, \mu, \Hi)$. As described in Section \ref{sec: representations of wreath}, we construct a representation $\pi'$ of $A \wr G$ on $L_2(X, \mu, \Hi)$. I claim that $\pi'$ is not strongly rigid. Indeed If it would be the case, we could conclude that $\eta$ is also strongly rigid, as we will show. Take a close to $\eta$ representation $\eta_1$ of $E$, going through the same steps, we get a representation $\pi'_1$ that is close to $\pi'$. Leveraging strong rigidity of $\pi'$ we get that there is a close to $Id_{L_2(X, \mu \Hi)}$ unitary operator $V$ that conjugates $\pi'$ and $\pi'_1$. By the definition of conjugating operator, we get that it commutes with $\pi'(h) = \pi'_1(h) = \iota(h)$ for every $h \in \bigoplus_G A$. So $V$ is totally fibered by Proposition \ref{prop: totally fibered operator}. It is now easy to check that its explicitly fibered form $(U_x)$ provides a conjugating operator for $\eta$ and $\eta_1$, such that $\lVert U_x - \Id_\Hi\rVert$ is small a.e.
\end{proof}

\section{Lamplighters are not strongly Ulam stable}

We say that a group $G$ is {\em strongly Ulam stable} if for every $\varepsilon > 0$ there exists $\delta > 0$ such that if $\rho$ is a unitary $\delta$-representation of $G$,  then there is a unitary representaion $\rho_1$ of $G$ such that $\lVert \rho_1 - \rho \lVert < \varepsilon$. We say that a group $G$ is {\em extensionally stronly Ulam stable} if for every $\varepsilon > 0$ there is $\delta > 0$ such that for every $\delta$ - representation $\rho$ of $G$ on a Hilber space $\Hi$ there is a Hilbert space $\Hi'$ containing $\Hi$, a unitary quasirepresentaion $\rho'$ of $G$ on $\Hi'$ whose restriction to $\Hi$ is $\rho$, and a unitary representation $\rho_1$ of $G$ on $\Hi'$ such that $\lVert \rho' - \rho_1 \rVert < \varepsilon$.

Let $G \curvearrowright (X, \mu)$ is an essentilally free measure preserving action of a countable group $G$ on a standard probability space $(X,\mu)$. We say that this action is {\em strongly Ulam stable} if for a $\varepsilon > 0$ there is $\delta > 0$ such that for any unitary $\delta$-representation $\rho$ of the action on a Hilbert space $\Hi$ there is  a unitary representation $\rho_1$ of the action on $\Hi$ such that $\lVert \rho - \rho_1 \rVert < \varepsilon$.

Let $E$ be a Borel equivalence relation on a standard probability space $(X, \mu)$. We would say that $E$ is strongly Ulam stable if for every $\varepsilon > 0$ there exists $\delta > 0$ such that for any unitary $\delta$-representation $\eta$ of $E$ on a Hilbert space $\Hi$ there is a unitary representation $\eta_1$ of $E$ on $\Hi$ such that $\lVert \eta_1 - \eta\rVert < \varepsilon$.

Now we establish connections of strong Ulam stability properties for various objects.


\begin{lem}\label{prop: ulam stability subgroup of equivalence}
Suppose $H$ is a countable group that is not extensionally strongly Ulam stable. Let $E$ be an ergodic equivalence relation on a standard probability space $(X,\mu)$ that orbitally contains an essentially free measure preserving action of $H$. Then $E$ is not strongly Ulam stable.
\end{lem}

\begin{proof}
By contradiction, assume by that $E$ is strongly Ulam stable, 
We would like to deduce that $H$ is extensionally stronly Ulam stable. 

Take a $\varepsilon>0$. There is a $\delta > 0$ such that for any $\delta$-quasirepresentation of $E$ there is an $\varepsilon$-close representation of $E$.

Let $\tau$ be a $\delta$-quasirepresentation of $H$ on a Hilbert space $\Hi$. Consider $\rho =(H \curvearrowright (X, \mu))\otimes  \tau $, a fibered quasirepresentation of $H$ over the action $H$ on $(X, \mu)$ that is  given by the orbit inclusion of $H$ into $E$. The latter could be considered a qusirepresentation of the orbit equivalence relation $E_H$ of the action $H \curvearrowright (X, \mu)$. We now use the induction construction to induce a quasirepresentation $\eta$ of $E$ on the Hilbert space $\Hi^{[E : E_H]}$. We can go back to get a quasirepresentation $\rho'$ of the action  $H \curvearrowright (X, \mu)$ on the Hilbert space
 $\Hi^{[E : E_H]}$. We remind that the subspace given by the $0$'th coordinate in the latter qusirepresentation is $E_H$-invariant, and $\rho$ is a sub-quasirepresentation of $\rho'$ (by restricting to the subspace described). We again go back to get a quasirepresentation $\tau'$ of $H$ on the Hilbert space $L_2(X, \mu) \otimes  \Hi^{[E : E_H]}$. Note that the latter contains an invariant subspace $L_2(X, \mu) \otimes  \Hi $, that in turn contains an invariant subspace $\Hi$ given by the embedding $v \mapsto 1 \otimes v$. 

 Quasireprsentation $\eta$  has the same defect as $\tau$. So there is a representation $\eta_1$ such that $\lVert \eta_1 - \eta \rVert < \varepsilon$.  Now we go back to get a representation $\rho_1$ of $H \curvearrowright (X, \mu)$ on $\Hi^{[E : E_H]}$ and then a representation $\tau_1$ of $H$ on the Hilbert space $\Hi^{[E : E_H]} \otimes L_2(X, \mu)$. it is now easy to see that $\tau_1$ is a representation that is $\varepsilon$-close to $\tau'$ and the latter is an extension of $\tau$, so we proved that $N$ is extensionally strongly Ulam stable.
 \end{proof}

We borrow the next piece from the Burger, Ozawa and Thom paper \cite{BOT13}. Namely, this is a combination of Lemma 2.6 from the latter paper (that rougly states that near a finite-dimensional almost invariant subspace there is an actual invariant subspace, so that projector to the latter subspace is close in the norm topology) and of a finite-dimensional  counterexample for Ulam-stability of $F_2$ from the preprint \cite{R09} due to P. Rolli.
\begin{prop}\label{prop: free is not stable}
The free group $F_2$ is not extensionally strongly Ulam stable.
\end{prop}

\begin{cor}\label{cor: oe of bernoulli ulam non-stable}
An orbit equivalence relation of a Bernoulli action of a non-amenable group is not strongly Ulam stable.
\end{cor}

Surprisingly, the proof of the next proposition will leverage deformation rigidity of amenable groups. Namely we will use the following
\begin{theor}[\cite{BOT13}, Theorem 4.3]
All countable unitarizable groups (in particular, all countable amenable) are deformation rigid.
\end{theor}

\begin{prop}\label{prop: ulam stability wreath product and equivalence}
Let $A$ be a non-trivial finite or countable abelian group and $G$ be any countable group. 
If $A \wr G$ is strongly Ulam stable, then there is a Bernoulli action of $G$ whose orbit equivalence relation $E$ is strongly Ulam stable. 
\end{prop}
\begin{proof}

Denote $\Gamma = A \wr G$. We use the construction of the Bernoulli action from section \ref{sec: representations of wreath}, so $E$ is the equivalence relation of the said action.

Take $\varepsilon > 0$, we want to prove that there is a $\delta > 0$ such that for every unitary $\delta$-representation $\eta$ of $E$ on a Hilbert space $\Hi$ there is a representation $\eta_1$ of $E$  such that $\lVert \eta_1 - \eta\rVert < \varepsilon$.

We use deformation rigidity of $N = \bigoplus_{G} A$ (it is abelian and hence amenable) to find $\delta_1$ such that if $\chi_1$ and $\chi_2$ are two representations  of $N$ with $\lVert \chi_1 - \chi_2\rVert < \delta_1$ then there is a unitary conjugating operator $W$ (i.e. $\chi_2(h) = W^{-1} \chi_1(h) W$ for every $h \in N$) such that $\lVert \Id - W\rVert < \varepsilon/4$. We use strong Ulam stability of $\Gamma$ to find a $\delta_2$ such that for every  $\delta_2$-representaion $\tau$ of $\Gamma$ there is a unitary representation $\tau_1$ such that $\lVert \tau - \tau_1\rVert < \min(\varepsilon/4, \delta_1)$.

Now, take $\delta > 0$ smaller than $\delta_1$, $\delta_2$, $\varepsilon/4$. Take any unitary $\delta$-representation $\eta$ of $E$ on $\Hi$. 
First note that $\eta$ corresponds to a $\delta$-representaion $\bar{\rho}$ of the action $G \curvearrowright (X,\mu)$ on $\Hi$. We then induce a $\delta$-representation $\rho$ of $G$ on $L_2(X,\mu, \Hi)$. 
From the latter we can induce a $\delta$-representation $\tau$ of $A \wr G$ on $L_2(X, \mu, \Hi)$ as we described earlier in Proposition \ref{prop: fibered to wreath}. 
By construction, there is a unitary representation $\tau_1$ of $\Gamma$ such that $\lVert \tau - \tau_1\rVert < \min(\varepsilon/4, \delta_1)$.  Now we denote $\chi_1 = \tau_1 \vert_N$. Since $\lVert \chi_1 - \iota \rVert < \delta_1$ (we remind that $\iota$ is the natural embedding of $L_\infty(X, \mu)$ into $\mathcal{B}(\Hi)$, and $N$ is identified with a subset of $L_\infty(X, \mu)$ by Pontryagin duality), there is a unitary conjugating operator $W$ with $\lVert W - \Id\rVert < \varepsilon/4$, such that $W^{-1} \chi_1(h) W = \iota$ for every $h \in N$. We now define a representation $\tau_ 2$ of $\Gamma$ by $\tau_2(\gamma) = W^{-1} \tau_1(\gamma) W$. it is easy to check that $\lVert \tau_2 - \tau \rVert < \varepsilon$ and that $\tau_2 \vert_N = \iota$. From the latter we may deduce, by Proposition \ref{prop: wreath to fibered}, that $\tau_2 \vert_G$ is a representation of $G$ that is fibered over the action $G \curvearrowright(X, \mu)$. From the fibered representation  $\tau_2 \vert_G$ we may construct a representation $\eta_2$ of $E$ and it is easy to check that $\lvert \eta_1 - \eta \rvert < \varepsilon$ (for this we use the fact that the norm of a fibered operator is the essential supremum of the norms of its fiber-components).

\end{proof}

\begin{proof}[Proof of Theorem \ref{thm: ulam non-stability}]
Assume by contradiction that $A \wr G$ is strongly Ulam-stable for a finite or countable abelian group $A$ and a non-amenable  countable $G$. By previous proposition, there is a non-trivial Bernoulli action of $G$ such that its orbit equivalence relation is strongly Ulam-stable, contrary to Corolary \ref{cor: oe of bernoulli ulam non-stable}.

\end{proof}

\section{Dense subgroups of full groups}

We remind that the {\em full group} $\llbracket E\rrbracket$ of an equivalence relation $E$ on a standard probability space $(X, \mu)$ is the group of all measure preserving automorphisms of $(X,\mu)$ whose graphs are subsets of $E$. The {\em uniform distance} between two elements $T_1, T_2$ of $\llbracket E\rrbracket$ is defined as the measure of the set where $T_1$ and $T_2$ differ; the topology defined by this metric is called the {\em uniform topology}.  A thorough treatment of full groups could be found in book \cite{Ke10}. 

A representation $\eta$ of equivalence relation $E$ naturally defines a continuous representation $\tilde{\eta}: \llbracket E\rrbracket \to \Aut(L_2(X, \mu, \Hi))$ of the full group of equivalence relation $E$, assuming the uniform topology on the full group and the strong operator topology on $\Aut(L_2(X, \mu, \Hi))$. This representation is given by the formula 
\[
\big( {(\tilde{\eta}(T))(f)} \big)(x) = \big(\eta(T^{-1}x, x)\big)(f(T^{-1}x)),
\]
for $x \in X$, $f \in L_2(X, \mu, \Hi)$, $T \in \llbracket E\rrbracket$.
We would say that a countable Borel measure preserving equivalence relation $E$ on a standard probability space $(X,\mu)$ orbitally contains a group $N$ if there is an essetially free action of $N$ on $(X,\mu)$ such that the $N$-orbit of almost every point is contained in the $E$-equivalence class of that point.

\begin{proof}[Proof of Theorem \ref{thm: dense subgroup}, the part concerning unitarizability]
We remind that we have an equivalence relation $E$ on a standard probability space $(X,\mu)$ that measurable contains a non-unitarizable group $N$. We want to prove that any countable dense subgoup $G$ of $\llbracket E\rrbracket$ is non-unitarizable. Assuming the contrary, we will show that $N$ is unitarizable. Let $\tau: N \to \Aut(\Hi)$ be a representation of $N$. This representation induces the fibered representation $\bar{\rho}$ of $N$, the tensor product $(N \curvearrowright X) \otimes \tau$. There is a corresponding representation $\eta$ of the orbit equivalence relation $E_N$ of the $N$-action. We induce a representation $\eta^E$ of equivalence realtion $E$. 
Consider the representation $\widetilde{\eta^{E}} : \llbracket E\rrbracket \to \Aut(L_2(X, \mu, \Hi^I))$, where $I = [0, [E:E_N]_{X,\mu})$ of the full group $\llbracket E \rrbracket$ obtained from $\eta^{E}$.
Since $G$ is unitarizable, there is a new scalar product on $L_2(X,\mu, \Hi^I)$ such that $\widetilde{\eta^{E}}(g)$ becomes unitary for every $g \in G$. The strong operator topology did not change since the new scalar product defines the same topology. Also, the unitaries are closed in the strong operator topology among invertible operators. Hence we get that $\widetilde{\eta^{E}}(\gamma)$ is unitary for any $\gamma \in \llbracket E\rrbracket$. Now, the orbit inclusion of $N$ into $E$ defines an embedding $\lambda: N \to \llbracket E\rrbracket$. It is easy to note that $\lambda \circ \widetilde{\eta^{E}}$ is a representation of $N$ that contains $\tau$ as a subrepresentation. But we showed that $\lambda \circ \widetilde{\eta^{E}}$ is unitarizable, so $\tau$ is unitarizable as well.
\end{proof}

\begin{proof}[Proof of Theorem \ref{thm: dense subgroup}, the part concerning rigidity]
	
Let $E$ be an equivalence relation on a standard probability space $(X, \mu)$ that orbitally contains a group $H$. Let $\tau$ be a non strongly rigid representation of $H$ on a Hilbert space $\Hi$. Denote $E_H$ the 

Consider the fibered  representation $\pi$   of $H$ on $L_2(X, \mu, \Hi)$  given by the tensor-product construction of $\tau$ with the action of $H$. Let $\bp$ be the corresponding representation of the action $H \curvearrowright (X, \mu)$. Denote $\eta$ the representation of equivalence relation $E_H$ obtained from the latter representation of the action $H \curvearrowright (X, \mu)$.  We then apply the induction construction to get a representation $\eta^E$ of equivalence relation $E$ on the space $\Hi^I$, where $I = [0, E : E_H)$. We then get a continuous (the uniform metric to the strong operator topology) representation $\widetilde{\eta^E}$ of $\llbracket E\rrbracket$ on $\Hi^I$. 
Lets identify $G$ with a dense subgroup of $\llbracket E\rrbracket$.
I claim of that representation of $G$ obtained by restriction, is not strongly rigid. Assuming the contrary, we may deduce that $\tau$ is strongly rigid. Indeed, Lets $\tau_1$ be a close to $\tau$ representation. We go through the same steps, and get a representation $\widetilde{\eta_1^E}$ of $E$. Let $\lambda$ be a natural embedding of $H$ into $\llbracket E\rrbracket$ (which we get from the action $H \curvearrowright (X, \mu)$). Notice that representation $\lambda \circ \widetilde{\eta^E}$ has an
Now, using rigidity of $\widetilde{\eta^E}\vert_G$, we get a conjugating unitary operator $U$ such that $U \widetilde{\eta^E}(g) = \widetilde{\eta_1^E}(g) U$ for all $g \in G$. Since $G$ is dense in $G$ and representaions in question are continuous, the latter equality holds for all $g \in \llbracket E\rrbracket$. Notice that $\tau$ is a subrepesentation of $\widetilde{\eta} \circ \lambda$ (the corresponding subspace is given by the composition of natural embeddings $\Hi \to L_2(X, \mu, \Hi) \to L_2(X, \mu, \Hi)^I$, denote $P_\Hi$ the projector from $L_2(X, \mu, \Hi)^I \simeq L_2(X, \mu, \Hi^I)$ to that copy of $\Hi$). Also, $\tau_1$ is a subrepresentation of $\widetilde{\eta_1^E}$, with the same invariant subspace. Now its not hard to see that the sign part of the polar decomposition of $P_\Hi U P_\Hi\vert_\Hi$ would be a conjugating unitary operator for representation $\tau$ and $\tau_1$, that is close to $\Id_\Hi$.

\end{proof}

\end{document}